\documentclass[a4paper]{amsart}

\usepackage[mathscr]{eucal}
\usepackage[svgnames]{xcolor}
\usepackage{amsmath,amssymb,amsbsy,amsthm,amsfonts,comment,tikz}
\usepackage{hyperref}
\hypersetup{
  pdfauthor = {Morimichi Kawasaki and Ryuma Orita},
  pdftitle = {Existence of pseudoheavy fibers of moment maps},
  pdfpagemode = UseNone,
  bookmarksnumbered=true,
}
\usepackage{graphicx}

\title{Existence of pseudoheavy fibers of moment maps}

\author{Morimichi Kawasaki} 
\address[Morimichi Kawasaki]{Research Institute for Mathematical Sciences, Kyoto University, Kyoto 606-8502, Japan}
\email{kawasaki@kurims.kyoto-u.ac.jp}

\author{Ryuma Orita} 
\address[Ryuma Orita]{Department of Mathematical Sciences, Tokyo Metropolitan University, Tokyo 192-0397, Japan}
\email{ryuma.orita@gmail.com}
\urladdr{https://ryuma-orita.github.io/}

\date{January 31, Reiwa 2}

\thanks{The first named author has been supported by IBS-R003-D1. This work has been supported by JSPS KAKENHI Grant Numbers JP18J00765, JP18J00335.}

\newtheorem{theorem}{Theorem}[section]
\newtheorem{lemma}[theorem]{Lemma}
\newtheorem{proposition}[theorem]{Proposition}
\newtheorem{corollary}[theorem]{Corollary}
\newtheorem{problem}[theorem]{Problem}

\theoremstyle{definition}
\newtheorem{definition}[theorem]{Definition}
\newtheorem{example}[theorem]{Example}

\theoremstyle{remark}
\newtheorem{remark}[theorem]{Remark}

\newcommand{\Ham}{\mathop{\mathrm{Ham}}\nolimits}
\newcommand{\QH}{\mathop{\mathrm{QH}}\nolimits}
\newcommand{\supp}{\mathop{\mathrm{supp}}\nolimits}
\newcommand{\Area}{\mathop{\mathrm{Area}}\nolimits}
\newcommand{\Arccos}{\mathop{\mathrm{Arccos}}\nolimits}
\newcommand{\RR}{\mathbb{R}}
\newcommand{\ZZ}{\mathbb{Z}}
\newcommand{\CP}{\mathbb{C}P}
\newcommand{\relmiddle}[1]{\mathrel{}\middle#1\mathrel{}}

\begin{document}

\begin{abstract}
In the present paper, we introduce the notion of pseudoheaviness of closed subsets of closed symplectic manifolds
and prove the existence of pseudoheavy fibers of moment maps.
In particular, we generalize Entov and Polterovich's theorem, which ensures the existence of non-displaceable fibers.
As its application, we provide a partial answer to a problem posed by them, which asks the existence of heavy fibers.
Moreover, we obtain a family of singular Lagrangian submanifolds in $S^2\times S^2$ with strange rigidities.
\end{abstract}

\maketitle

\tableofcontents


\section{Introduction}\label{intro}

\subsection{Backgrounds}

Let $(M,\omega)$ be a closed symplectic manifold.
Let $C(M)$ (resp.\ $C^{\infty}(M)$) denote the set of continuous (resp.\ smooth) functions on $M$.
Given a finite-dimensional Poisson-commutative subspace $\mathbb{A}$ of $C^{\infty}(M)$,
the \textit{moment map} $\Phi\colon M\to\mathbb{A}^{\ast}$ is given by $F(x)=\langle\Phi(x),F\rangle$ for $x\in M$ and $F\in\mathbb{A}$.

A subset $X$ of $M$ is called \textit{displaceable} from a subset $Y\subset M$
if there exists a Hamiltonian $H\colon S^1\times M\to\mathbb{R}$ such that $\varphi_H(X)\cap\overline{Y}=\emptyset$,
where $\varphi_H$ is the \textit{Hamiltonian diffeomorphism} generated by $H$
(i.e., the time-1 map of the isotopy $\{\varphi_H^t\}_{t\in\RR}$ associated with the \textit{Hamiltonian vector field} $X_H$ defined by the formula $\iota_{X_{H_t}}\omega=-dH_t$ where $H_t=H(t,\cdot)$ for $t\in S^1=\RR/\mathbb{Z}$)
and $\overline{Y}$ is the topological closure of $Y$.
Otherwise, $X$ is called \textit{non-displaceable} from $Y$.

Since Gromov's famous work \cite{G}, it has been an important problem in symplectic geometry to find non-displaceable subsets.
Biran, Entov, and Polterovich \cite{BEP} proved that the standard moment map on the complex projective space has only one non-displaceable fiber using the Calabi quasi-morphism constructed in \cite{EP03}.
Entov and Polterovich \cite{EP06} generalized that argument and proved the following theorem.

\begin{theorem}[{\cite[Theorem 2.1]{EP06}, see also \cite[Theorem 6.1.8]{PR}}]\label{existence of non-disp fiber}
Let $(M,\omega)$ be a closed symplectic manifold and $\mathbb{A}$ a finite-dimensional Poisson-commutative subspace of $C^{\infty}(M)$.
Then, there exists $y_0\in\Phi(M)$ such that $\Phi^{-1}(y_0)$ is non-displaceable from itself.
\end{theorem}

To prove Theorem \ref{existence of non-disp fiber}, Entov and Polterovich \cite{EP06} introduced the concept of partial symplectic quasi-state $($see Definition \ref{def:psqs}$)$.
In \cite{EP09}, they introduced the notion of heaviness of closed subsets in terms of partial symplectic quasi-states.

\begin{definition}[{\cite[Definition 1.3]{EP09}}]
Let $\zeta\colon C(M)\to\mathbb{R}$ be a partial symplectic quasi-state on $(M,\omega)$.
A closed subset $X$ of $M$ is said to be \textit{$\zeta$-heavy} (resp.\ \textit{$\zeta$-superheavy}) if
\[
	\zeta(H)\geq\inf_X H \quad \left(\text{resp.}\ \zeta(H)\leq\sup_X H\right)
\]
for any $H\in C(M)$.
\end{definition}

Here we collect properties of (super)heavy subsets.

\begin{theorem}[{\cite[Theorem 1.4]{EP09}}]\label{prop:shv is hv}
Let $\zeta\colon C(M)\to\mathbb{R}$ be a partial symplectic quasi-state on $(M,\omega)$.
\begin{enumerate}
\item Every $\zeta$-superheavy subset is $\zeta$-heavy.
\item Every $\zeta$-heavy subset is non-displaceable from itself.
\item Every $\zeta$-heavy subset is non-displaceable from every $\zeta$-superheavy subset.
In particular, every $\zeta$-heavy subset intersects every $\zeta$-superheavy subset.
\end{enumerate}
\end{theorem}

Entov and Polterovich posed the following problem relating to Theorem \ref{existence of non-disp fiber}.

\begin{problem}[{\cite[Section 1.8.2]{EP09}, see also \cite[Question 4.9]{E}}]\label{existence of heavy}
Let $(M,\omega)$ be a closed symplectic manifold and $\mathbb{A}$ a finite-dimensional Poisson-commutative subspace of $C^{\infty}(M)$.
Let $\zeta\colon C(M)\to\mathbb{R}$ be a partial symplectic quasi-state on $(M,\omega)$ made from the Oh--Schwarz spectral invariant $($see \cite{Sch,Oh05}$)$.
Then, does there exist $y_0\in\Phi(M)$ such that $\Phi^{-1}(y_0)$ is $\zeta$-heavy?
\end{problem}




\subsection{Main results}\label{sec:results}

Let $(M,\omega)$ be a closed symplectic manifold.
For an open subset $U$ of $M$,
let $\mathcal{H}(U)$ be the subset of $C(M)$ consisting of all functions supported in $U$.
We introduce the notion of \textit{pseudoheaviness} of closed subsets.

\begin{definition}
Let $\zeta\colon C(M)\to\mathbb{R}$ be a partial symplectic quasi-state on $(M,\omega)$.
A closed subset $X$ of $M$ is said to be \textit{$\zeta$-pseudoheavy}
if for any open neighborhood $U$ of $X$ there exists a function $F\in\mathcal{H}(U)$ such that $\zeta(F)>0$.
\end{definition}

By definition, every $\zeta$-heavy subset is $\zeta$-pseudoheavy.
The following proposition tells us the reason why we call such subsets pseudoheavy (compare Theorem \ref{prop:shv is hv}).

\begin{proposition}[\cite{KO}]\label{not pseudoheavy}
Let $\zeta\colon C(M)\to\mathbb{R}$ be a partial symplectic quasi-state on $(M,\omega)$.
If a closed subset $X$ of $M$ is $\zeta$-pseudoheavy,
then $X$ is non-displaceable from itself and from every $\zeta$-superheavy subset.
\end{proposition}

We prove Proposition \ref{not pseudoheavy} in Section \ref{sec:proof}.
Our main theorem is the following one which asserts the existence of a pseudoheavy fiber instead of heavy one.

\begin{theorem}[Main Theorem]\label{existence of pseudoheavy}
Let $(M,\omega)$ be a closed symplectic manifold and $\mathbb{A}$ a finite-dimensional Poisson-commutative subspace of $C^{\infty}(M)$.
Let $\zeta\colon C(M)\to\mathbb{R}$ be a partial symplectic quasi-state on $(M,\omega)$.
Then, there exists $y_0\in\Phi(M)$ such that $\Phi^{-1}(y_0)$ is $\zeta$-pseudoheavy.
\end{theorem}

As written in Proposition \ref{not pseudoheavy},
any $\zeta$-pseudoheavy subset is non-displaceable from any $\zeta$-superheavy subset.
Thus, we can see Theorem \ref{existence of pseudoheavy} as a relative version of Theorem \ref{existence of non-disp fiber}.
For another relative version of Theorem \ref{existence of non-disp fiber},
see \cite{Ka2}.

In Section \ref{sec:example}, we will provide examples of closed subsets which are pseudoheavy, but not heavy.
Moreover, we will point out that the positive answer to Problem \ref{existence of heavy} does not hold for a general partial symplectic quasi-state.

In Section \ref{sec:simple}, we introduce a notion of simplicity of partial symplectic quasi-states (Definition \ref{def:simple}) and prove the following proposition.

\begin{proposition}\label{simple qs deha phv ha hv}
Let $\zeta\colon C(M)\to\mathbb{R}$ be a simple partial symplectic quasi-state on $(M,\omega)$.
Then, every $\zeta$-pseudoheavy subset is $\zeta$-heavy.
\end{proposition}

As an application of Theorem \ref{existence of pseudoheavy} and Proposition \ref{simple qs deha phv ha hv}, we have the following corollary which gives a partial answer to Problem \ref{existence of heavy}.

\begin{corollary}\label{existence of hv fiber in simple case}
Let $(M,\omega)$ be a closed symplectic manifold and $\mathbb{A}$ a finite-dimensional Poisson-commutative subspace of $C^{\infty}(M)$.
Let $\zeta\colon C(M)\to\mathbb{R}$ be a simple partial symplectic quasi-state on $(M,\omega)$.
Then, there exists $y_0\in\Phi(M)$ such that $\Phi^{-1}(y_0)$ is $\zeta$-heavy.
\end{corollary}

For a closed symplectic manifold $(M,\omega)$, let $\zeta_M$ denote the partial symplectic quasi-state on $(M,\omega)$ made from
the Oh--Schwarz spectral invariant with respect to the fundamental class $[M]$ of the quantum homology $\QH_\ast(M;\mathbb{Z}/2\mathbb{Z})$ of $M$.
Let $(S,\omega)$ be a closed Riemann surface $S$ with the symplectic form $\omega$.
Then, it is known that $\zeta_S(F^2)=\max\{\zeta_S(F)^2,\zeta_S(-F)^2\}$ for any function $F\colon S\to\RR$ (see \cite{EP06} for genus zero case, \cite{HLS} for positive genus case).
Under this condition, one can prove that the partial symplectic quasi-state $\zeta_S$ is simple.
Among experts, it is an open conjecture for many years that every symplectic quasi-state (see \cite{EP06} for the definition) made from the Oh--Schwarz spectral invariant is always simple.
However, it is known to be difficult to prove that these symplectic quasi-states are actually simple when the dimension of $M$ is greater than two.

In Section \ref{sec:CAM}, we obtain some singular Lagrangian submanifolds in $S^2\times S^2$ with strange rigidities.
To prove that strange rigidities, we use Theorem \ref{theorem:AVSC stem is shv} which is the key theorem for proving Theorem \ref{existence of pseudoheavy}.
To be more precise, we define functions $J_1,H^1\colon S^2\times S^2\to\RR$ by
\[
	J_1(p)=z_1+z_2%
	\quad\text{and}\quad%
	H^1(p)=x_1x_2+y_1y_2+z_1z_2,
\]
for each $p=(x_1,y_1,z_1,x_2,y_2,z_2)\in S^2\times S^2\subset\RR^3\times \RR^3$, respectively.
Then, $\{J_1,H^1\}=0$ (see Section \ref{sec:CAM}) and the integrable system $(J_1,H^1)$ is called (a special case of) the \textit{coupled angular momenta} \cite{SZ,LP,HP}.
We set $\Phi_1^1=(J_1,H^1)\colon S^2\times S^2\to\RR^2$.
Given $c\in [-1,-1/2]$, we set the Lagrangian submanifold $L_c=(\Phi_1^1)^{-1}(0,c)$ of $(S^2\times S^2,\omega_1)$ (see Section \ref{sec:CAM} for the definition of $\omega_1$ and more details).
Let $Z$ be a topological space obtained by pinching two disjoint meridians in the 2-torus $T^2$ as shown in \textsc{Figure} \ref{fig:torus}.
Then, we have the following result.

\begin{figure}[h]
\centering
\includegraphics[width=7truecm]{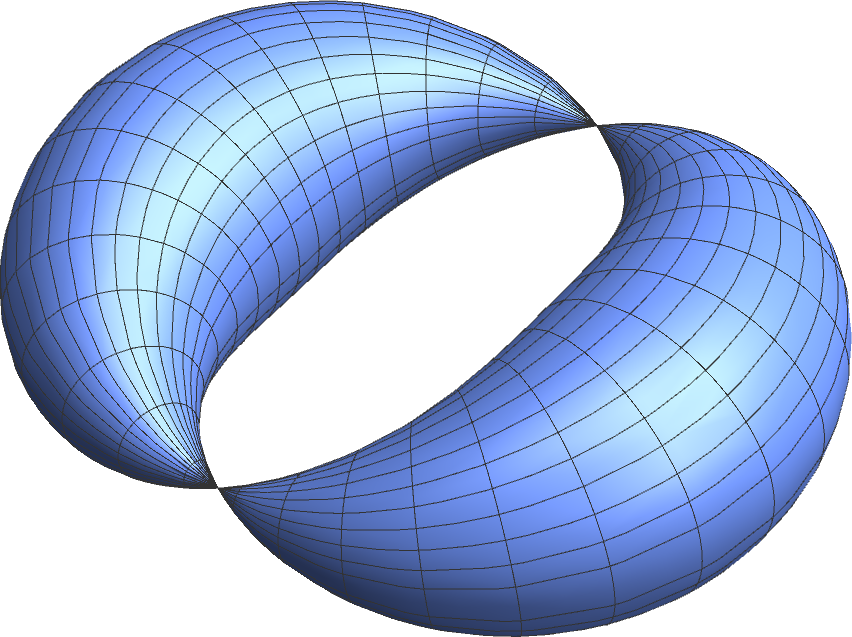}
\caption{A doubly pinched torus}
\label{fig:torus}
\end{figure}

\begin{theorem}\label{intro memo}
There exists a family $\{Z_c\}_{c\in [-1,-1/2]}$ of closed subsets of $S^2\times S^2$ such that for any $c\in[-1,-1/2]$
\begin{enumerate}
\item $Z_c$ is homeomorphic to $Z$ if $c\neq -1$,
\item $Z_c$ is non-displaceable from $Z_{c'}$ and from $L_d$ in $(S^2\times S^2,\omega_1)$ for any $c'\in[-1,-1/2]$ and any $d\in [-1,c]$,
\item $Z_c$ is displaceable from $L_d$ in $(S^2\times S^2,\omega_1)$ for any $d\in (c,-1/2]$.
\end{enumerate}
\end{theorem}

We show Theorem \ref{intro memo} as a corollary of Theorem \ref{thm:between} in Section \ref{sec:proof of strange}.
In order to prove Theorem \ref{intro memo}, we use the following partial symplectic quasi-states.
For $c=-1$, Entov and Polterovich \cite{EP09} constructed a partial symplectic quasi-state $\zeta_{-1}$ on $(S^2\times S^2,\omega_1)$
such that the Lagrangian sphere $L_{-1}$ is $\zeta_{-1}$-superheavy.
For every $c\in (-1,-1/2]$, Fukaya, Oh, Ohta, and Ono \cite{FOOO} constructed a partial symplectic quasi-state $\zeta_c$ on $(S^2\times S^2,\omega_1)$
such that the Lagrangian torus $L_c$ is $\zeta_c$-superheavy (see also \cite{ElP,LZ} for the case $c=-1/2$).

Moreover, we deal with generalized coupled angular momenta (see \cite{HP,Ga,P}) in Section \ref{sec:CAM}.
We prove that some of them have at least two non-displaceable fibers (Corollary \ref{cor:FOOO-OU-EP}) using Proposition \ref{not pseudoheavy} and Theorem \ref{existence of pseudoheavy}.
Furthermore, we prove that another generalized coupled angular momentum has only one non-displaceable fiber (Theorem \ref{HP example}).


\section{Proof of Theorem \ref{existence of pseudoheavy}}\label{sec:proof}

In this section, we provide the definition of partial symplectic quasi-state and proofs of Proposition \ref{not pseudoheavy} and Theorem \ref{existence of pseudoheavy}.
Let $(M,\omega)$ be a closed symplectic manifold.
Let $\Ham(M,\omega)$ denote the group of Hamiltonian diffeomorphisms of $(M,\omega)$.

\subsection{Partial symplectic quasi-states and pseudoheaviness}

\begin{definition}\label{def:psqs}
A \textit{partial symplectic quasi-state} on $(M,\omega)$ is a functional $\zeta\colon C(M)\to\mathbb{R}$ satisfying the following conditions.
\begin{description}
	\item[Normalization] $\zeta(a)=a$ for any constant function $a$.
	\item[Stability] For any $H_1,H_2\in C(M)$
	\[
		\min_M(H_1-H_2)\leq\zeta(H_1)-\zeta(H_2)\leq\max_M(H_1-H_2).
	\]
	In particular, \textbf{Monotonicity} holds: $\zeta(H_1)\leq\zeta(H_2)$ if $H_1\leq H_2$.
	\item[Semi-homogeneity] $\zeta(sH)=s\zeta(H)$ for any $H\in C(M)$ and any $s>0$.
	\item[Hamiltonian Invariance] $\zeta(H\circ\phi)=\zeta(H)$ for any $H\in C(M)$ and any $\phi\in\Ham(M,\omega)$.
	\item[Vanishing] $\zeta(H)=0$ for any $H\in C(M)$ whose support is displaceable from itself.
	\item[Quasi-subadditivity] $\zeta(H_1+H_2)\leq\zeta(H_1)+\zeta(H_2)$ for any $H_1,H_2\in C^{\infty}(M)$ satisfying $\{H_1,H_2\}=0$.
\end{description}
\end{definition}

\begin{remark}
In this paper, we adopted the properties listed in \cite[Section 4.5]{PR} as the definition of partial symplectic quasi-state.
There are different definitions of partial symplectic quasi-state as in \cite[Section 4]{EP06} and \cite[Definition 13.3]{FOOO}.
One can confirm that our definition is more general than the latter.
In addition, we note that the quasi-subadditivity is called ``the triangle inequality" in \cite[Theorem 3.6]{EP09} and \cite[Definition 13.3]{FOOO}.
\end{remark}

First we prove Proposition \ref{not pseudoheavy}.

\begin{proof}[Proof of Proposition \ref{not pseudoheavy}]
Let $X$ be a $\zeta$-pseudoheavy subset of $M$.
Assume, on the contrary, that $X$ is displaceable either from itself or from a $\zeta$-superheavy subset $Y$.
Then, there exists an open neighborhood $U$ of $X$ that is displaceable either from itself or from $Y$.
If $U$ is displaceable from itself, then by the vanishing of $\zeta$, $\zeta(F)=0$ for any $F\in\mathcal{H}(U)$.
This contradicts the $\zeta$-pseudoheaviness of $X$.

If $U$ is displaceable from $Y$,
then we can choose $\phi\in\Ham(M,\omega)$ such that $\phi(U)\cap Y=\emptyset$.
Since $Y$ is $\zeta$-superheavy, for any $F\in\mathcal{H}\bigl(\phi(U)\bigr)$
\[
	\zeta(F)\leq\sup_{Y}F=0.
\]
By the Hamiltonian invariance of $\zeta$, it means that $\zeta(G)\leq 0$ for any $G\in\mathcal{H}(U)$.
This contradicts the $\zeta$-pseudoheaviness of $X$.
Therefore, $X$ is non-displaceable from itself and from every $\zeta$-superheavy subset.
\end{proof}


\subsection{Proof of Theorem \ref{existence of pseudoheavy}}

Let $\zeta\colon C(M)\to\mathbb{R}$ be a partial symplectic quasi-state on $(M,\omega)$.
We need the following proposition to prove Theorem \ref{theorem:stem is shv}.

\begin{proposition}[{\cite[Proposition 4.1]{EP09}}]\label{prop:hv iff}
Let $X$ be a closed subset of $M$.
\begin{enumerate}
\item $X$ is $\zeta$-heavy if and only if $\zeta(H)=0$ for any $H\in C(M)$ satisfying $H\leq 0$ and $H|_X\equiv 0$.
\item $X$ is $\zeta$-superheavy if and only if $\zeta(H)=0$ for any $H\in C(M)$ satisfying $H\geq 0$ and $H|_X\equiv 0$.
\end{enumerate}
\end{proposition}

Given a finite-dimensional Poisson-commutative subspace $\mathbb{A}$ of $C^{\infty}(M)$,
we recall that the moment map $\Phi\colon M\to\mathbb{A}^{\ast}$ is given by $F(x)=\langle\Phi(x),F\rangle$ for $x\in M$ and $F\in\mathbb{A}$.
We define \textit{NPH-stems} which generalize stems introduced in \cite{EP06}.

\begin{definition}\label{def:NPH-stem}
A closed subset $X$ of $M$ is called a \textit{$\zeta$-NPH-stem} (resp.\ \textit{stem})
if there exists a finite-dimensional Poisson-commutative subspace $\mathbb{A}$ of $C^{\infty}(M)$ satisfying the following conditions.
\begin{enumerate}
\item $X=\Phi^{-1}(p)$ for some $p\in\Phi(M)$.
\item Every non-trivial fiber of $\Phi$, other than $X$, is not $\zeta$-pseudoheavy (resp.\ is displaceable from itself).
\end{enumerate}
\end{definition}

Here NPH stands for ``non-pseudoheavy.''
By Proposition \ref{not pseudoheavy}, every stem is a $\zeta$-NPH-stem for any partial symplectic quasi-state $\zeta$.
A crucial property of stems is the following.

\begin{theorem}[{\cite[Theorem 1.8]{EP09}}]\label{theorem:stem is shv}
Every stem is $\zeta$-superheavy for any partial symplectic quasi-state $\zeta$.
\end{theorem}

We generalize Theorem \ref{theorem:stem is shv} as follows.

\begin{theorem}\label{theorem:AVSC stem is shv}
Every $\zeta$-NPH-stem is $\zeta$-superheavy.
\end{theorem}

The proof of Theorem \ref{theorem:AVSC stem is shv} is almost parallel to that of \cite[Theorem 1.8]{EP09}, but we use the quasi-subadditivity instead of the partial quasi-additivity (see, for example, \cite[Section 4.6]{PR}).

\begin{proof}[Proof of Theorem \ref{theorem:AVSC stem is shv}]
Let $X=\Phi^{-1}(p)$, $p\in\Phi(M)$, be a $\zeta$-NPH-stem.
Take any function $H\colon\mathbb{A}^{\ast}\to\mathbb{R}$ which vanishes on an open neighborhood $V$ of $p$.
First we claim that $\zeta(\Phi^{\ast}H)\leq 0$.

Consider a finite open cover $\mathcal{U}=\{U_1,\ldots,U_d\}$ of $\Phi(M)\setminus V$ so that
each $\Phi^{-1}(U_i)$ contains no $\zeta$-pseudoheavy fiber.
Take a partition of unity $\{\rho_1,\ldots,\rho_d\}$ subordinated to $\mathcal{U}$.
Namely, $\sum_{i=1}^d\rho_i|_{\Phi(M)\setminus V}\equiv 1$ and $\supp(\rho_i)\subset U_i$ for any $i$.
Since $\supp\bigl(\Phi^{\ast}(\rho_iH)\bigr)\subset\Phi^{-1}(U_i)$, by the definition of pseudoheaviness,
\[
	\zeta(\Phi^{\ast}(\rho_iH)\bigr)\leq 0
\]
for any $i$. Since $\{\Phi^{\ast}(\rho_iH),\Phi^{\ast}(\rho_jH)\}=0$ for any $i$ and $j$,
by the quasi-subadditivity,
\[
	\zeta(\Phi^{\ast}H)=\zeta\left(\sum_{i=1}^d\Phi^{\ast}(\rho_iH)\right)\leq\sum_{i=1}^d\zeta\bigl(\Phi^{\ast}(\rho_iH)\bigr)\leq 0,
\]
and this completes the proof of the claim.

Now given any function $G\in C(M)$ satisfying $G\geq 0$ and $G|_X\equiv 0$,
one can find a function $H\colon\mathbb{A}^{\ast}\to\mathbb{R}$ and an open neighborhood $V$ of $p$ with $H|_V\equiv 0$ such that $G\leq\Phi^{\ast}H$.
By the normalization, the monotonicity and the above claim,
\[
	0=\zeta(0)\leq\zeta(G)\leq\zeta(\Phi^{\ast}H)\leq 0.
\]
Hence $\zeta(G)=0$.
By Proposition \ref{prop:hv iff} (2), $X$ is $\zeta$-superheavy.
\end{proof}

Now we are in a position to prove our main theorem (Theorem \ref{existence of pseudoheavy}).

\begin{proof}[Proof of Theorem \ref{existence of pseudoheavy}]
Arguing by contradiction, assume that every fiber of $\Phi$ is not $\zeta$-pseudoheavy.
Then, every fiber is a $\zeta$-NPH-stem.
Hence, by Theorem \ref{theorem:AVSC stem is shv}, every fiber is $\zeta$-superheavy.
Since all fibers are mutually disjoint, it contradicts Theorem \ref{prop:shv is hv} (1) and (3).
\end{proof}

\begin{remark}
In the proof of Theorem \ref{existence of pseudoheavy}, we do not use the vanishing property of $\zeta$.
\end{remark}


\section{Examples of pseudoheavy, but not heavy fibers}\label{psd not heavy}\label{sec:example}

Here we provide examples of moment maps with no heavy fiber.

\begin{proposition}\label{borman and fooo}
Let $(M,\omega)$ be a closed symplectic manifold and $\mathbb{A}$ a finite-dimensional Poisson-commutative subspace of $C^{\infty}(M)$.
Let $\zeta_1,\zeta_2\colon C(M)\to\RR$ be partial symplectic quasi-states on $(M,\omega)$.
Assume that there exist $y_1,y_2\in\Phi(M)$ such that $y_1\neq y_2$ and $\Phi^{-1}(y_i)$ is $\zeta_i$-superheavy for $i=1,2$.
Then, the functional $\frac{1}{2}(\zeta_1+\zeta_2)\colon C(M)\to\RR$ is also a partial symplectic quasi-state on $(M,\omega)$ and
\begin{enumerate}
\item $\Phi^{-1}(y_1)\cup\Phi^{-1}(y_2)$ is $\frac{1}{2}(\zeta_1+\zeta_2)$-superheavy.
\item $\Phi^{-1}(y_1)$ and $\Phi^{-1}(y_2)$ are $\frac{1}{2}(\zeta_1+\zeta_2)$-pseudoheavy.
\item $\Phi$ does not admit any $\frac{1}{2}(\zeta_1+\zeta_2)$-heavy fiber.
\end{enumerate}
\end{proposition}

For examples of $(M,\omega)$, $\zeta_1,\zeta_2$ and $\Phi$ satisfying the assumption listed in Proposition \ref{borman and fooo}, see
\cite{FOOO}, \cite[Theorem 1.2]{B} and \cite{V18}.

Let $\Sigma_2$ be a closed Riemann surface of genus 2 with an area form $\omega$.
Rosenberg \cite{R} constructed a functional $\zeta_P\colon C(\Sigma_2)\to\RR$ from Py's Calabi quasi-morphism defined in \cite{Py6}
and proved that $\zeta_P$ is a partial symplectic quasi-state on $(\Sigma_2,\omega)$.

Let $F_P\colon\Sigma_2\to\RR$ be a generic Morse function with exactly six critical points as shown in \textsc{Figure} \ref{fig:Sigma_2}.
Let $p_1,\ldots,p_6$ be the critical points of $F_P$ such that $c_1<\cdots<c_6$, where $c_i=F_P(p_i)$ for $i=1,\ldots,6$.

\begin{figure}[h]
\centering
\includegraphics[width=7truecm]{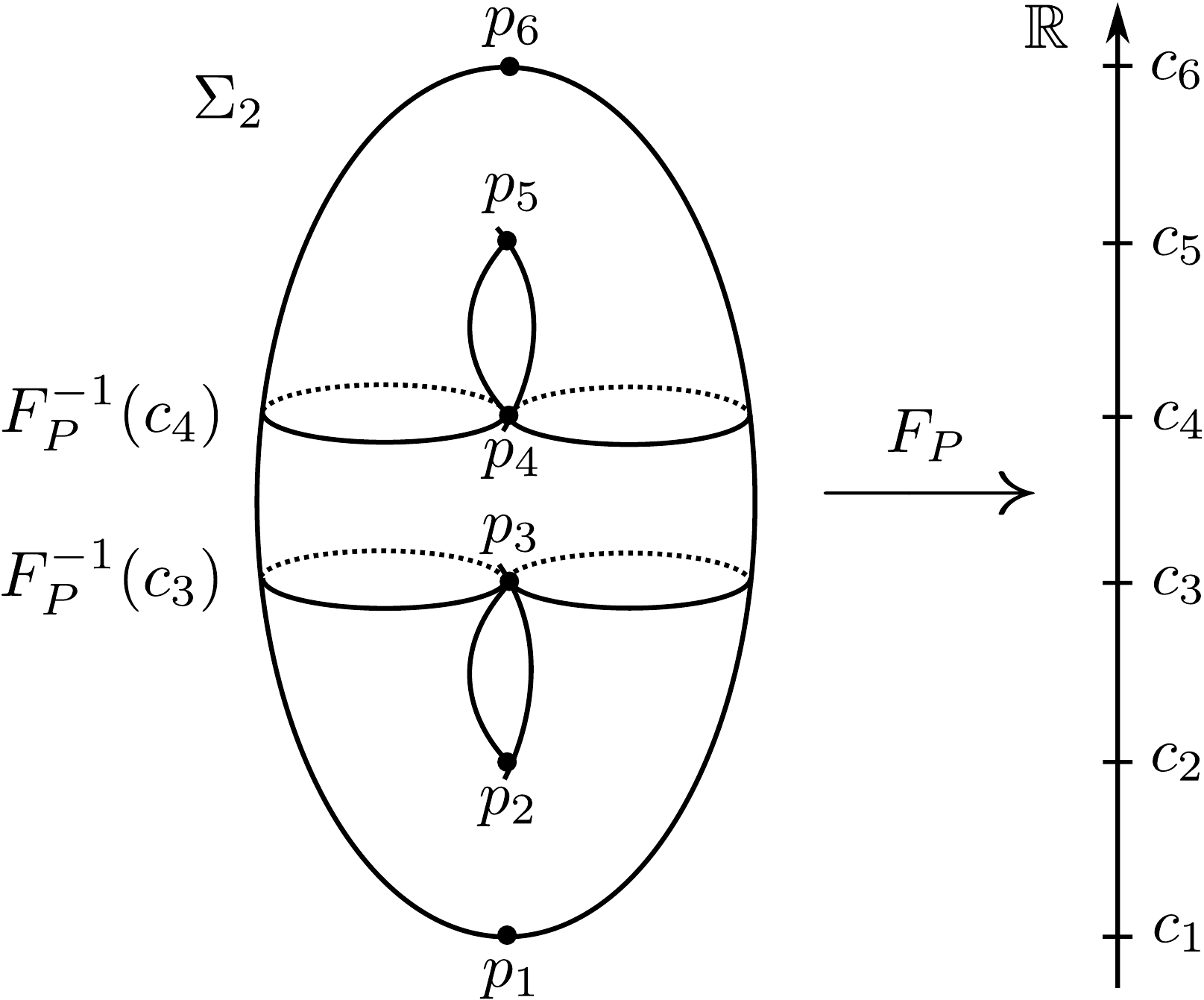}
\caption{A generic Morse function on $\Sigma_2$ with exactly six critical points}
\label{fig:Sigma_2}
\end{figure}

\begin{proposition}\label{py pseudoheavy}
Let $\zeta_P$ as above.
Then,
\begin{enumerate}
	\item $F_P^{-1}(c_3)\cup F_P^{-1}(c_4)$ is $\zeta_P$-superheavy.
	\item $F_P^{-1}(c_3)$ and $F_P^{-1}(c_4)$ are $\zeta_P$-pseudoheavy.
	\item $F_P$ does not admit any $\zeta_P$-heavy fiber.
\end{enumerate}
\end{proposition}

To prove Propositions \ref{borman and fooo} and \ref{py pseudoheavy}, we use the following lemma.

\begin{lemma}\label{general fooo}
Let $(M,\omega)$ be a closed symplectic manifold
and $\mathbb{A}$ a finite-dimensional Poisson-commutative subspace of $C^{\infty}(M)$.
Let $\zeta\colon C(M)\to\RR$ be a partial symplectic quasi-state on $(M,\omega)$.
Assume that there exist $y_1,y_2\in\Phi(M)$ such that $y_1\neq y_2$
and $\zeta(f\circ\Phi)=\frac{1}{2}\bigl(f(y_1)+f(y_2)\bigr)$ for any continuous function $f\in C(\mathbb{A}^{\ast})$.
Then, 
\begin{enumerate}
\item $\Phi^{-1}(y_1)\cup\Phi^{-1}(y_2)$ is $\zeta$-superheavy.
\item $\Phi^{-1}(y_1)$ and $\Phi^{-1}(y_2)$ are $\zeta$-pseudoheavy.
\item $\Phi$ does not admit any $\zeta$-heavy fiber.
\end{enumerate}
\end{lemma}

\begin{proof}
We fix an isomorphism $\mathbb{A}^{\ast}\cong\RR^k$ for some $k$.
For a continuous function $H\colon M\to\RR$, take a continuous function $\underline{H}\colon \RR^k\to\RR$ such that $H\leq\underline{H}\circ\Phi$ and
\[
	\max\{\underline{H}(y_1),\underline{H}(y_2)\}\leq\sup_{\Phi^{-1}(y_1)\cup\Phi^{-1}(y_2)}H.
\]
Then, by the monotonicity of $\zeta$ and the assumption,
\[
	\zeta(H)\leq\zeta(\underline{H}\circ\Phi)%
	=\frac{1}{2}\bigl(\underline{H}(y_1)+\underline{H}(y_2)\bigr)%
	\leq\sup_{\Phi^{-1}(y_1)\cup\Phi^{-1}(y_2)}H.
\]
Since $H$ is arbitrary, we complete the proof of (1).

We show that $\Phi^{-1}(y_1)$ is $\zeta$-pseudoheavy.
For any open neighborhood $U$ of $\Phi^{-1}(y_1)$ choose an open neighborhood $V$ of $y_1$ such that $y_2\notin V$ and $\Phi^{-1}(V)\subset U$.
Take a function $f\in\mathcal{H}(V)$ such that $f(y_1)>0$.
Then,
\[
	\zeta(f\circ\Phi)=\frac 12f(y_1)>0.
\]
Since $f\circ\Phi\in\mathcal{H}(U)$, $\Phi^{-1}(y_1)$ is $\zeta$-pseudoheavy.
Similarly, we can prove that $\Phi^{-1}(y_2)$ is also $\zeta$-pseudoheavy.
This completes the proof of (2).

Let $y\in\Phi(M)\setminus\{y_1,y_2\}$.
Since $\Phi^{-1}(y)$ is disjoint from the $\zeta$-superheavy subset $\Phi^{-1}(y_1)\cup\Phi^{-1}(y_2)$,
by Theorem \ref{prop:shv is hv} (3),
$\Phi^{-1}(y)$ is not $\zeta$-heavy.
We show that $\Phi^{-1}(y_1)$ is not $\zeta$-heavy.
Since $y_1\neq y_2$, we can choose a function $\underline{G}\in C(\RR^k)$ such that $\underline{G}(y_1)=1$ and $\underline{G}(y_2)=0$.
Set $G=\underline{G}\circ\Phi\in C(M)$.
Then,
\[
	\zeta(G)
	=\frac 12\bigl(\underline{G}(y_1)+\underline{G}(y_2)\bigr)
	=\frac 12<1=\inf_{\Phi^{-1}(y_1)}G.
\]
Hence $\Phi^{-1}(y_1)$ is not $\zeta$-heavy.
Similarly, we can prove that $\Phi^{-1}(y_2)$ is also not $\zeta$-heavy.
This completes the proof of Lemma \ref{general fooo}.
\end{proof}

\begin{proof}[Proof of Proposition \ref{borman and fooo}]
To confirm that $\frac 12(\zeta_1+\zeta_2)$ is a partial symplectic quasi-state on $(M,\omega)$,
we only check the stability since other properties follow from the definition.
Since $\zeta_i$ ($i=1,2$) satisfies the stability,
for any $H_1,H_2\in C(M)$
\[
	\min_M(H_1-H_2)\leq\zeta_i(H_1)-\zeta_i(H_2)\leq\max_M(H_1-H_2).
\]
By summing up with $i=1,2$ and dividing by 2,
\[
	\min_M(H_1-H_2)\leq\frac 12\bigl(\zeta_1(H_1)+\zeta_2(H_1)\bigr)-\frac 12\bigl(\zeta_1(H_2)+\zeta_2(H_2)\bigr)\leq\max_M(H_1-H_2).
\]
Hence $\frac 12(\zeta_1+\zeta_2)$ also satisfies the stability.

Now we claim that for any continuous function $f\in C(\mathbb{A}^{\ast})$
\[
	\frac 12(\zeta_1+\zeta_2)(f\circ\Phi)=\frac 12\bigl(f(y_1)+f(y_2)\bigr).
\]
Indeed, since $\Phi^{-1}(y_i)$ ($i=1,2$) is $\zeta_i$-superheavy,
by Theorem \ref{prop:shv is hv} (1),
\[
	f(y_i)=\inf_{\Phi^{-1}(y_i)}f\circ\Phi\leq\zeta_i(f\circ\Phi)\leq\sup_{\Phi^{-1}(y_i)}f\circ\Phi=f(y_i).
\]
Thus, $\zeta_i(f\circ\Phi)=f(y_i)$ and this shows the claim.
Hence Lemma \ref{general fooo} yields Proposition \ref{borman and fooo}.
\end{proof}

To prove Proposition \ref{py pseudoheavy}, we use the following theorem which is a special case of Py's theorem.

\begin{theorem}[{A special case of \cite[Th\'eor\`eme 2]{Py6}, see also \cite[Theorem 4.4]{R}}]\label{py calculate}
Let $(\Sigma_2,\omega)$, $\zeta_P$ and $F_P$ as above.
Then,
$\zeta_P(H)=\frac{1}{2}\bigl(H(p_3)+H(p_4)\bigr)$
for any smooth function $H\colon\Sigma_2\to\RR$ with $\{H,F_P\}=0$.
\end{theorem}

\begin{proof}[Proof of Proposition \ref{py pseudoheavy}]
For any smooth function $f\colon\RR\to\RR$,
since $\{f\circ F_P,F_P\}=0$,
Theorem \ref{py calculate} implies that
\[
	\zeta_P(f\circ F_P)=\frac{1}{2}\bigl((f\circ F_P)(p_3)+(f\circ F_P)(p_4)\bigr)=\frac{1}{2}\bigl(f(c_3)+f(c_4)\bigr).
\]
By the stability of $\zeta_P$, this equality still holds for any continuous function $f$.
Thus, $(\Sigma_2,\omega)$, $\zeta_P$, $F_P$, $c_3$ and $c_4$ satisfy the assumption of Lemma \ref{general fooo}.
Hence Proposition \ref{py pseudoheavy} follows from Lemma \ref{general fooo}.
\end{proof}


\section{Simple partial symplectic quasi-states}\label{sec:simple}

Let $(M,\omega)$ be a closed symplectic manifold.
Let $\zeta\colon C(M)\to\RR$ be a partial symplectic quasi-state on $(M,\omega)$.
For a closed subset $X$ of $M$, we define a real number $\tau_{\zeta}(X)$ by
\[
	\tau_{\zeta}(X)=\inf\{\,\zeta(a)\mid a\colon M\to[0,1],\ a|_X\equiv 1\,\}.
\]

\begin{remark}
When $\zeta$ is a symplectic quasi-state in the sense of \cite{EP06} or, more generally,
a quasi-state in the sense of \cite{A91}, then the above $\tau_{\zeta}$ is a quasi-measure \cite{A91}.
\end{remark}

\begin{definition}
\label{def:simple}
A partial symplectic quasi-state $\zeta$ on $(M,\omega)$ is called \textit{simple} if
$\tau_{\zeta}(X)=0$ or $\tau_{\zeta}(X)=1$ for any closed subset $X$ of $M$.
\end{definition}

\begin{remark}
When $\zeta$ is a quasi-state, then our definition of simplicity is equivalent to that of \cite{A93}.
\end{remark}

To prove Proposition \ref{simple qs deha phv ha hv}, we use the following lemmas.

\begin{lemma}\label{open nbhd of phv}
Let $\zeta$ be a partial symplectic quasi-state on $(M,\omega)$ and $X$ a closed subset of $M$.
If $X$ is $\zeta$-pseudoheavy,
then for any open neighborhood $U$ of $X$, $\tau_{\zeta}(\overline{U})>0$.
\end{lemma}

\begin{proof}
By the definition of pseudoheaviness, there exists a function $F\in\mathcal{H}(U)$ such that $\zeta(F)>0$.
By the monotonicity and the normalization of $\zeta$,
$\max_M{F}\geq\zeta(F)>0$.
For any continuous function $a\colon M\to[0,1]$ with  $a|_{\overline{U}}\equiv 1$,
by the monotonicity and the semi-homogeneity of $\zeta$,
\[
	\zeta(F)\leq\zeta\left(\max_M{F}\cdot a\right)=\max_M{F}\cdot\zeta(a).
\]
Thus, by the definition of $\tau_{\zeta}$, 
\[
	\tau_{\zeta}(\overline{U})\geq\left(\max_M{F}\right)^{-1}\cdot\zeta(F)>0.\qedhere
\]
\end{proof}

\begin{lemma}\label{tau one is heavy}
Let $\zeta$ be a partial symplectic quasi-state on $(M,\omega)$ and $X$ a closed subset of $M$.
Then, $\tau_{\zeta}(X)=1$ if and only if $X$ is $\zeta$-heavy.
\end{lemma}

\begin{proof}
Assume that $\tau_{\zeta}(X)=1$.
Let $G\colon M\to\RR$ be a continuous function such that $G>0$.
Define a continuous function $\underline{G}\colon M\to\RR$ by
\[
	\underline{G}(x)=\min\left\{G(x),\inf_XG\right\}.
\]
Then, the function
$(\inf_XG)^{-1}\cdot \underline{G}$ takes values in $[0,1]$ and $(\inf_XG)^{-1}\cdot \underline{G}|_X\equiv 1$.
Thus, by the definition of $\tau_{\zeta}$ and the semi-homogeneity of $\zeta$,
\[
	\inf_XG\cdot\tau_{\zeta}(X)\leq\inf_XG\cdot\zeta\left(\left(\inf_XG\right)^{-1}\cdot \underline{G}\right)=\zeta(\underline{G}).
\]
By the monotonicity of $\zeta$ and $\tau_{\zeta}(X)=1$,
\[
	\inf_XG=\inf_XG\cdot\tau_{\zeta}(X)\leq\zeta(\underline{G})\leq\zeta(G).
\]
Let $H\colon M\to\RR$ be a continuous function.
Take a positive number $r$ so that $H+r>0$.
Then, the above argument yields $\inf_X(H+r)\leq\zeta(H+r)$.
By the stability of $\zeta$,
\[
	\zeta(H+r)-\zeta(H)\leq r.
\]
Therefore,
\[
	\inf_XH=\inf_X(H+r)-r\leq\zeta(H+r)-r\leq\zeta(H).
\]
Since $H$ is arbitrary, $X$ is $\zeta$-heavy.

Conversely, assume that $X$ is $\zeta$-heavy.
Let $a\colon M\to[0,1]$ be a continuous function with $a|_X\equiv 1$.
By the monotonicity and the normalization of $\zeta$, $\zeta(a)\leq 1$.
On the other hand, since $X$ is $\zeta$-heavy,
\[
	\zeta(a)\geq\inf_Xa=1,
\]
which concludes that $\zeta(a)=1$.
Since $a$ is arbitrary, $\tau_\zeta(X)=1$.
\end{proof}

\begin{lemma}\label{lem:nbd is hv}
Let $\zeta$ be a partial symplectic quasi-state on $(M,\omega)$ and $X$ a closed subset of $M$.
If the closure of any sufficiently small open neighborhood of $X$ is $\zeta$-heavy $($resp.\ $\zeta$-superheavy$)$,
then $X$ is also $\zeta$-heavy $($resp.\ $\zeta$-superheavy$)$.
\end{lemma}

\begin{proof}
Let $H\in C(M)$.
For any $\varepsilon>0$ choose an open neighborhood $U$ of $X$ whose closure is $\zeta$-heavy so that
\[
	0\leq\inf_XH-\inf_{\overline{U}}H\leq\varepsilon.
\]
Since $\overline{U}$ is $\zeta$-heavy,
\[
	\zeta(H)\geq\inf_{\overline{U}}H\geq\inf_XH-\varepsilon.
\]
Since $\varepsilon$ is arbitrary,
$\zeta(H)\geq\inf_XH$ for all $H\in C(M)$.
Thus $X$ is $\zeta$-heavy.
We can prove the case of superheaviness similarly.
\end{proof}

Now we are in a position to prove Proposition \ref{simple qs deha phv ha hv}.

\begin{proof}[Proof of Proposition \ref{simple qs deha phv ha hv}]
Let $X$ be a $\zeta$-pseudoheavy subset of $M$.
Let $U$ be an open neighborhood of $X$.
Then, by Lemma \ref{open nbhd of phv}, $\tau_{\zeta}(\overline{U})>0$.
Since $\zeta$ is simple, $\tau_{\zeta}(\overline{U})=1$.
Hence by Lemma \ref{tau one is heavy}, $\overline{U}$ is $\zeta$-heavy.
Since the closure of any open neighborhood of $X$ is $\zeta$-heavy,
by Lemma \ref{lem:nbd is hv}, $X$ is also $\zeta$-heavy.
\end{proof}

We can prove the converse of Proposition \ref{simple qs deha phv ha hv}.

\begin{proposition}
Let $\zeta$ be a partial symplectic quasi-state on $(M,\omega)$ such that every $\zeta$-pseudoheavy subset is $\zeta$-heavy.
Then, $\zeta$ is simple.
\end{proposition}

\begin{proof}
Choose arbitrary closed subset $X$ of $M$ such that $\tau_{\zeta}(X)>0$.
By the definitions of pseudoheaviness and $\tau_{\zeta}$, $X$ is $\zeta$-pseudoheavy.
By the assumption, $X$ is $\zeta$-heavy and thus, by Lemma \ref{tau one is heavy}, $\tau_{\zeta}(X)=1$.
Since $X$ is arbitrary, $\zeta$ is simple.
\end{proof}

\section{Generalized coupled angular momenta}\label{sec:CAM}

In this section, we provide applications of
Theorems \ref{existence of pseudoheavy} and \ref{theorem:AVSC stem is shv}.
Let
\[
	S^2=\left\{\,(x,y,z)\in\RR^3\relmiddle| x^2+y^2+z^2=1\,\right\}
\]
be the two-sphere with the standard symplectic form $\omega_{S^2}$.
We consider the product $S^2\times S^2$ with the symplectic form $\omega_R=\mathrm{pr}_1^\ast\omega_{S^2}+R(\mathrm{pr}_2^\ast\omega_{S^2})$,
where $R$ is a positive number and
$\mathrm{pr}_1, \mathrm{pr}_2\colon S^2\times S^2\to S^2$ are the first and second projections, respectively.
Let $f\colon[-1,1]^2\to\RR$ be a smooth function.
We define functions $J_R,H_f\colon S^2\times S^2\to\RR$ by the formulas
\begin{equation}\label{eq:J}
	J_R(x_1,y_1,z_1,x_2,y_2,z_2)=z_1+Rz_2,
\end{equation}
\begin{equation}\label{eq:H}
	H_f(x_1,y_1,z_1,x_2,y_2,z_2)=x_1x_2+y_1y_2+z_1z_2-f(z_1,z_2),
\end{equation}
respectively, and set $\Phi_{R,f}=(J_R,H_f)\colon S^2\times S^2\to\RR^2$.
Since the function $H_f$ is conserved along the Hamiltonian vector field $X_{J_R}$ associated to $J_R$, Noether's theorem implies that $J_R$ and $H_f$ are Poisson-commutative on $(S^2\times S^2,\omega_R)$.

Let $s\in\RR$.
We set $H^s=H_f$ and $\Phi_R^s=\Phi_{R,f}$ when $f(z_1,z_2)=(1-s)z_1z_2$.
Namely, for each $(x_1,y_1,z_1,x_2,y_2,z_2)\in S^2\times S^2$,
\begin{equation}\label{eq:Hs}
	H^s(x_1,y_1,z_1,x_2,y_2,z_2)=x_1x_2+y_1y_2+sz_1z_2.
\end{equation}


\subsection{Non-displaceable fibers of $\Phi_{R,f}$}

In the following corollary of Theorem \ref{existence of pseudoheavy} and Proposition \ref{not pseudoheavy}, we set $R=1$.

\begin{corollary}\label{cor:FOOO-OU-EP}
If $\|f\|_{L^\infty}<1/4$,
then $\Phi_{1,f}$ has at least two non-displaceable fibers.
\end{corollary}

\begin{proof}
Recall $L_{-1/2}=(\Phi_1^1)^{-1}(0,-1/2)$ and $L_{-1}=(\Phi_1^1)^{-1}(0,-1)$ in $S^2\times S^2$.
As pointed out in Secition \ref{sec:results}, $L_c$ is $\zeta_c$-superheavy for each $c=-1/2,-1$.
By Theorem \ref{existence of pseudoheavy}, for each $c=-1/2,-1$
there exists $w_c\in\Phi_{1,f}(S^2\times S^2)$ such that $\Phi_{1,f}^{-1}(w_c)$ is $\zeta_c$-pseudoheavy.
Since $L_c$ is $\zeta_c$-superheavy,
by Proposition \ref{not pseudoheavy},
\begin{equation}\label{eq:disjoint}
	\Phi_{1,f}^{-1}(w_c)\cap L_c\neq\emptyset.
\end{equation}

Let $(x_1,y_1,z_1,x_2,y_2,z_2)\in S^2\times S^2$.
Since $\lvert z_1\rvert\leq 1$ and $\lvert z_2\rvert\leq 1$,
\[
	\left\lvert H^1(x_1,y_1,z_1,x_2,y_2,z_2)-H_f(x_1,y_1,z_1,x_2,y_2,z_2)\right\rvert
	=\lvert f(z_1,z_2)\rvert
	\leq\|f\|_{L^{\infty}}<\frac{1}{4}.
\]
Thus,
\[
	H_f(L_{-1/2})\subset\left(-\frac 12-\frac 14,-\frac 12+\frac 14\right)=\left(-\frac 34,-\frac 14\right)
\]
and
\[
	H_f(L_{-1})\subset\left(-1-\frac 14,-1+\frac 14\right)=\left(-\frac 54,-\frac 34\right).
\]
Since $J_1(L_{-1/2})=J_1(L_{-1})=\{0\}$,
\[
	\Phi_{1,f}(L_{-1/2})\subset\{0\}\times\left(-\frac 34,-\frac 14\right)%
	\quad\text{and}\quad%
	\Phi_{1,f}(L_{-1})\subset\{0\}\times\left(-\frac 54,-\frac 34\right).
\]
Therefore, by \eqref{eq:disjoint}, $w_{-1/2}\in\{0\}\times (-3/4,-1/4)$ and $w_{-1}\in\{0\}\times (-5/4,-3/4)$.
Since $(-5/4,-3/4)\cap (-3/4,-1/4)=\emptyset$,
we conclude that $\Phi_{1,f}^{-1}(w_c)$, $c=-1/2,-1$, are mutually disjoint $\zeta_c$-pseudoheavy subsets, respectively.
In particular, by Proposition \ref{not pseudoheavy},
the map $\Phi_{1,f}$ has at least two non-displaceable fibers.
\end{proof}

On the other hand, some $\Phi_{R,f}$ has only one non-displaceable fiber.
More precisely, we have the following theorem.
For a positive number $R$ and a smooth function $f\colon [-1,1]^2\to\RR$, we define a function $F_{R,f}\colon[-1,1]\to\RR$ by
\[
	F_{R,f}(z)=-\frac 12\bigl(f(-Rz,z)+f(Rz,-z)+2Rz^2\bigr).
\]

\begin{theorem}\label{HP example}
Let $R$ be a positive number and $f\colon [-1,1]^2\to\RR$ a smooth function such that $F_{R,f}\equiv 0$.
Then, $\Phi_{R,f}^{-1}(0,0)$ is a stem,
in particular, by Theorem \ref{theorem:stem is shv}, $\zeta$-superheavy for any partial symplectic quasi-state $\zeta$ on $(S^2\times S^2,\omega_R)$.
\end{theorem}

Theorem \ref{HP example} immediately follows from the following lemma.
We define a diffeomorphism $\psi$ of $S^2\times S^2$ by
\[
	\psi(x_1,y_1,z_1,x_2,y_2,z_2)=(-x_1,y_1,-z_1,x_2,-y_2,-z_2).
\]
We note that $\psi$ is a Hamiltonian diffeomorphism of $(S^2\times S^2,\omega_R)$ for any $R>0$.

\begin{lemma}\label{prop:displace}
Let $R$ be a positive number and $f\colon [-1,1]^2\to\RR$ a smooth function.
If $(a,b)\notin\{0\}\times [m_{R,f},M_{R,f}]$,
then $\psi$ displaces $\Phi_{R,f}^{-1}(a,b)$ from itself,
where $M_{R,f}$ and $m_{R,f}$ are the maximum and the minimum of the function $F_{R,f}$, respectively.
\end{lemma}

\begin{example}\label{example:stem}
Assume that $f(z_1,z_2)=(1-s)z_1z_2$ where $s\geq 0$.
Then, $F_{R,f}(z)=-sRz^2$.
Since $-1\leq z\leq 1$, $m_{R,f}=-sR$ and $M_{R,f}=0$.
By Lemma \ref{prop:displace},
the fiber $(\Phi_R^s)^{-1}(a,b)$ is displaceable from itself whenever $(a,b)\notin\{0\}\times [-sR,0]$,
where $\Phi_R^s=(J_R,H^s)$ (recall \eqref{eq:Hs} for the definition of $H^s$).
Moreover, Theorem \ref{HP example} means that $(\Phi_R^0)^{-1}(0,0)$ is a stem.
\end{example}

\begin{proof}[Proof of Lemma \ref{prop:displace}]
Since
\[
	J_R\bigl(\psi(p)\bigr)%
	=-z_1-Rz_2=-J_R(p)
\]
for any $p=(x_1,y_1,z_1,x_2,y_2,z_2)\in S^2\times S^2$,
\[
	\Phi_{R,f}\Bigl(\psi\bigl(\Phi_{R,f}^{-1}(a,b)\bigr)\Bigr)\subset\{-a\}\times\RR.
\]
Hence $\psi$ displaces $\Phi_{R,f}^{-1}(a,b)$ from itself whenever $a\neq 0$.

Thus, we consider the case $a=0$.
Let $b\in\RR$ and $p=(x_1,y_1,z_1,x_2,y_2,z_2)\in\Phi_{R,f}^{-1}(0,b)$.
Then, $z_1=-Rz_2$ and $H_f(p)=b$.
Hence
\begin{align*}
	H_f\bigl(\psi(p)\bigr)%
	&=-x_1x_2-y_1y_2+z_1z_2-f(-z_1,-z_2)\\
	&=-H_f(p)+2z_1z_2-f(z_1,z_2)-f(-z_1,-z_2)\\
	&=-b-2Rz_2^2-f(-Rz_2,z_2)-f(Rz_2,-z_2).
\end{align*}
By the definitions of $M_{R,f}$ and $m_{R,f}$,
\[
	-b+2m_{R,f}\leq H_f\bigl(\psi(p)\bigr)\leq -b+2M_{R,f}.
\]
Therefore, since $J_R\Bigl(\psi\bigl(\Phi_{R,f}^{-1}(0,b)\bigr)\Bigr)=\{0\}$,
\[
	\Phi_{R,f}\Bigl(\psi\bigl(\Phi_{R,f}^{-1}(0,b)\bigr)\Bigr)\subset
	\{0\}\times[-b+2m_{R,f},-b+2M_{R,f}].
\]
Hence $\psi$ displaces $\Phi_{R,f}^{-1}(0,b)$ from itself whenever $b\notin [-b+2m_{R,f},-b+2M_{R,f}]$,
equivalently, $b\notin [m_{R,f},M_{R,f}]$.

As a consequence, $\psi$ displaces $\Phi_{R,f}^{-1}(a,b)$ from itself
whenever $(a,b)\notin\{0\}\times [m_{R,f},M_{R,f}]$.
\end{proof}

We set
\[
	\aleph=\inf\{\,\|f\|_{L^{\infty}}\mid \Phi_{1,f}\ \text{has only one non-displaceable fiber}\,\}.
\]
By Corollary \ref{cor:FOOO-OU-EP} and Theorem \ref{HP example}, $1/4\leq\aleph\leq1$.
It is an interesting problem to determine the exact value of $\aleph$.


\subsection{Proof of Theorem \ref{intro memo}}\label{sec:proof of strange}

We set $M=(S^2\setminus\{N,S\})^2\subset S^2\times S^2$ where $N=(0,0,1)$ and $S=(0,0,-1)$.
The Hamiltonian circle action generated by the function $J_1\colon S^2\times S^2\to\RR$ is free on the regular level set $(J_1|_M)^{-1}(0)$ (recall \eqref{eq:J} for the definition of $J_1$).
Then, the quotient manifold $(J_1|_M)^{-1}(0)/S^1$ carries a symplectic form $\bar{\sigma}$ such that $\bar{\tau}^*\bar{\sigma}=\iota^*\omega_1$,
where
\[
	\bar{\tau}\colon (J_1|_M)^{-1}(0)\to (J_1|_M)^{-1}(0)/S^1%
	\quad\text{and}\quad%
	\iota\colon (J_1|_M)^{-1}(0)\hookrightarrow S^2\times S^2
\]
are the projection and the inclusion, respectively (see \cite{MW} for details, see also \cite[Section 1.7]{PR}).

Let $(z,\theta)$ denote the coordinates of the annulus $(-1,1)\times\RR/2\pi\ZZ$ and set $\sigma=(4\pi)^{-1}dz\wedge d\theta$.
Eliashberg and Polterovich \cite{ElP} implicitly constructed a symplectomorphism $\phi\colon\bigl((J_1|_M)^{-1}(0)/S^1,\bar{\sigma}\bigr)\to\bigl((-1,1)\times\RR/2\pi\ZZ,\sigma\bigr)$ such that
\[
	\tau(x_1,y_1,z,x_2,y_2,-z)=(z,\theta),
\]
for each $(x_1,y_1,z,x_2,y_2,-z)\in (J_1|_M)^{-1}(0)\subset S^2\times S^2$,
where $\tau=\phi\circ\bar{\tau}$ and $\theta$ is the angle between $(x_1,y_1)$ and $(x_2,y_2)$ in $\RR^2$.

Let $s,b\in\RR$ be real numbers satisfying $0\leq s\leq1$ and $-s<b\leq 0$.
We set
\[
	\alpha(s,b)=\left\{\,(z,\theta)\in (-1,1)\times\RR/2\pi\ZZ\relmiddle|z^2=\frac{\cos{\theta}-b}{\cos{\theta}+s}\,\right\}.
\]
Then $\alpha(s,b)$ is a contractible simple closed curve in the annulus $(-1,1)\times\RR/2\pi\ZZ$.
Moreover,
\begin{equation}\label{eq:reduction}
	(\Phi_1^s)^{-1}(0,b)=\iota\left(\tau^{-1}\bigl(\alpha(s,b)\bigr)\right),
\end{equation}
where $\Phi_1^s=(J_1,H^s)\colon S^2\times S^2\to\RR^2$ (recall \eqref{eq:Hs} for the definition of $H^s$).

Let $D(s,b)$ denote the open disk bounded by $\alpha(s,b)$.
Then, the area of $D(s,b)$ with respect to $\sigma$ is given by
\[
	\Area_{\sigma}{\bigl(D(s,b)\bigr)}%
	=\frac 1\pi\int_0^{\Arccos{b}}\sqrt{\frac{\cos{\theta}-b}{\cos{\theta}+s}}\,d\theta,
\]
where $\Arccos{b}\in [0,\pi]$.

We consider the case $b=-s$ where $0\leq s\leq 1$.
We define subsets of $(-1,1)\times\RR/2\pi\ZZ$ by
\[
	\mathcal{A}_s=(-1,1)\times q\bigl(\{\pm\Arccos(-s)\}\bigr),
\]
\[
	\mathcal{D}_s=(-1,1)\times q\bigl([-\Arccos(-s),\Arccos(-s)]\bigr),
\]
where $q\colon \RR\to\RR/2\pi\ZZ$ is the natural quotient map.
We note that
\[
	\mathcal{A}_s=\tau\left((\Phi_1^s)^{-1}(0,-s)\cap M\right).
\]
For convenience, we set $\alpha(s,-s)=\mathcal{A}_s$ and $D(s,-s)=\mathcal{D}_s$  (see \textsc{Figure} \ref{fig:annulus}).
Since we have
\[
	\Area_{\sigma}\bigl(D(s,-s)\bigr)=\frac 1{4\pi}\int_{-\Arccos(-s)}^{\Arccos(-s)}\left(\int_{-1}^1dz\right)d\theta=\frac 1\pi\Arccos(-s),
\]
fixing $s\in [0,1]$, the function $\Area_{\sigma}{\bigl(D(s,b)\bigr)}$ on $b\in [-s,0]$ is continuous and strictly monotone decreasing.
 
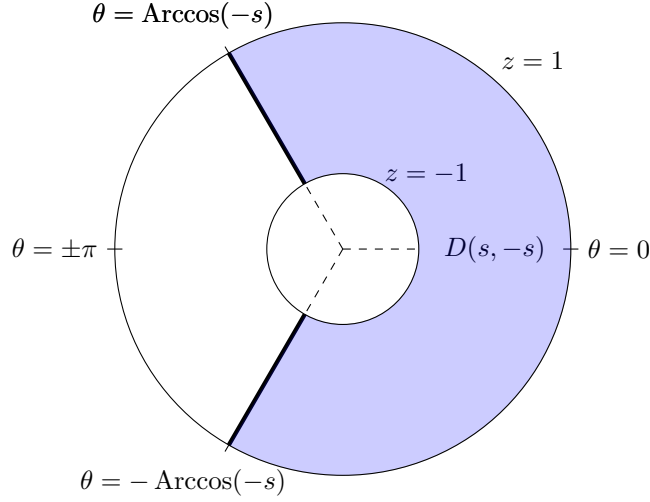
\begin{figure}[h]
\centering
\begin{tikzpicture}
\draw (4,4) circle [x radius=1cm, y radius=1cm];
\draw (4,4) circle [x radius=3cm, y radius=3cm];
\draw[ultra thick] (4-0.5,4+0.866) -- ++(120:2cm);
\draw[ultra thick] (4-0.5,4-0.866) -- ++(-120:2cm);
\draw (6.9,4) -- (7.1,4);
\draw (0.9,4) -- (1.1,4);
\draw (4-1.5,4+2.598) -- ++(120:1mm);
\draw (4-1.5,4-2.598) -- ++(-120:1mm);
\draw[dashed] (4,4) -- ++(0:1cm);
\draw[dashed] (4,4) -- ++(-120:1cm);
\draw[dashed] (4,4) -- ++(120:1cm);
\node[right] at (7.1,4) {$\theta=0$};
\node[left] at (0.9,4) {$\theta=\pm\pi$};
\node at (1.9,7.1) {$\theta=\Arccos(-s)$};
\node at (1.9,0.9) {$\theta=-\Arccos(-s)$};
\node at (1.9,7.1) {$\theta=\Arccos(-s)$};
\node at (6,4) {$D(s,-s)$};
\node at (6.5,6.5) {$z=1$};
\node at (5.1,5) {$z=-1$};
\filldraw [fill=blue, opacity=.2] (4-0.5,4-0.866) -- (4-1.5,4-2.598) arc (-120:120:3) -- (4-0.5,4+0.866) arc (120:-120:1);
\end{tikzpicture}
\caption{The case $s=1/2$. The subset $\alpha(s,-s)$ is the union of the two thick lines and $D(s,-s)$ is the blue-colored area.}
\label{fig:annulus}
\end{figure}

We have the following result on the partial symplectic quasi-states $\zeta_c$, $c\in [-1,-1/2]$, on $(S^2\times S^2,\omega_1)$ introduced in Section \ref{sec:results}.

\begin{theorem}\label{thm:between}
For any $c\in [-1,-1/2]$ the fiber
$(\Phi_1^{s_c})^{-1}(0,-s_c)$ is $\zeta_d$-superheavy for any $d\in [-1,c]$,
and is not $\zeta_d$-superheavy for any $d\in (c,-1/2]$,
where $s_c=-\cos{\left(\pi\Area_{\sigma}{\bigl(D(1,c)\bigr)}\right)}$.
\end{theorem}

To prove Theorem \ref{thm:between}, we use the following proposition.

\begin{proposition}[see, for example, the proof of {\cite[Lemma 3.1]{AM}}]\label{prop:lift}
Let $X$ be a subset of $(-1,1)\times\RR/2\pi\ZZ$ and $\psi\in\Ham\bigl((-1,1)\times\RR/2\pi\ZZ,\sigma\bigr)$.
Then, there exists $\bar{\psi}\in\Ham(S^2\times S^2,\omega_1)$ such that $\supp{\bar{\psi}}\subset M$
and $\bar{\psi}\left(\iota\bigl(\tau^{-1}(X)\bigr)\right)=\iota\left(\tau^{-1}\bigl(\psi(X)\bigr)\right)$.
\end{proposition}

\begin{proof}[Proof of Theorem \ref{thm:between}]
Let $c\in [-1,-1/2]$.
Since the function $b\mapsto\Area_{\sigma}{\bigl(D(1,b)\bigr)}$ is strictly monotone decreasing,
\[
	\frac 12=\Area_{\sigma}\bigl(D(1,-1/2)\bigr)\leq\Area_{\sigma}{\bigl(D(1,c)\bigr)}\leq\Area_{\sigma}{\bigl(D(1,-1)\bigr)}=1.
\]
Therefore, $0\leq s_c\leq 1$.
Moreover,
\[
	\Area_{\sigma}{\bigl(D(s_c,-s_c)\bigr)}=\frac 1\pi\Arccos(-s_c)=\Area_{\sigma}{\bigl(D(1,c)\bigr)}.
\]

Let $d\in [-1,c]$.
If $b\in (-s_c,0]$, then
\[
	\Area_{\sigma}{\bigl(D(s_c,b)\bigr)}<\Area_{\sigma}{\bigl(D(s_c,-s_c)\bigr)}=\Area_{\sigma}{\bigl(D(1,c)\bigr)}\leq\Area_{\sigma}{\bigl(D(1,d)\bigr)}.
\]
Thus, the contractible simple closed curve $\alpha(s_c,b)$ is displaceable from $\alpha(1,d)$ in the annulus $(-1,1)\times\RR/2\pi\ZZ$.
Namely, we can choose a Hamiltonian diffeomorphism $\psi$ of $\bigl((-1,1)\times\RR/2\pi\ZZ,\sigma\bigr)$ such that
\begin{equation}\label{eq:scc disp}
	\psi\bigl(\alpha(s_c,b)\bigr)\cap\alpha(1,d)=\emptyset.
\end{equation}
Then, applying Proposition \ref{prop:lift} for $\alpha(s_c,b)$ yields that there exists a Hamiltonian diffeomorphism $\bar{\psi}$ of $(S^2\times S^2,\omega_1)$ such that $\supp{\bar{\psi}}\subset M$
and
\begin{equation}\label{eq:fiber disp}
	\bar{\psi}\left(\iota\left(\tau^{-1}\bigl(\alpha(s_c,b)\bigr)\right)\right)=\iota\left(\tau^{-1}\left(\psi\bigl(\alpha(s_c,b)\bigr)\right)\right).
\end{equation}
By \eqref{eq:reduction},
$(\Phi_1^{s_c})^{-1}(0,b)=\iota\left(\tau^{-1}\bigl(\alpha(s_c,b)\bigr)\right)\subset M$.
Note that $\iota\left(\tau^{-1}\bigl(\alpha(1,d)\bigr)\right)=L_d\cap M$.
Since $\supp{\bar{\psi}}\subset M$, using \eqref{eq:scc disp} and \eqref{eq:fiber disp},
\begin{align*}
	\bar{\psi}\left((\Phi_1^{s_c})^{-1}(0,b)\right)\cap L_d%
	&=\bar{\psi}\left((\Phi_1^{s_c})^{-1}(0,b)\cap M\right)\cap L_d\\
	&=\bar{\psi}\left((\Phi_1^{s_c})^{-1}(0,b)\right)\cap (L_d\cap M)\\
	&=\bar{\psi}\left(\iota\left(\tau^{-1}\bigl(\alpha(s_c,b)\bigr)\right)\right)\cap\iota\left(\tau^{-1}\bigl(\alpha(1,d)\bigr)\right)\\
	&=\iota\left(\tau^{-1}\left(\psi\bigl(\alpha(s_c,b)\bigr)\right)\right)\cap\iota\left(\tau^{-1}\bigl(\alpha(1,d)\bigr)\right)\\
	&=\iota\left(\tau^{-1}\left(\psi\bigl(\alpha(s_c,b)\bigr)\cap\alpha(1,d)\right)\right)\\
	&=\emptyset.
\end{align*}
Therefore, $(\Phi_1^{s_c})^{-1}(0,b)$ is displaceable from $L_d$ in $(S^2\times S^2,\omega_1)$.
Since $L_d$ is $\zeta_d$-superheavy,
Proposition \ref{not pseudoheavy} implies that $(\Phi_1^{s_c})^{-1}(0,b)$ is not $\zeta_d$-pseudoheavy.
Moreover, by Example \ref{example:stem}, $(\Phi_1^{s_c})^{-1}(a,b)$ is displaceable from itself
whenever $(a,b)\not\in\{0\}\times [-s_c,0]$.
Therefore, the fiber $(\Phi_1^{s_c})^{-1}(0,-s_c)$ is a $\zeta_d$-NPH-stem,
and hence, is $\zeta_d$-superheavy by Theorem \ref{theorem:AVSC stem is shv}.

Let $d\in (c,-1/2]$.
Since the function $b\mapsto\Area_{\sigma}{\bigl(D(s_c,b)\bigr)}$ is continuous and strictly monotone decreasing,
there uniquely exists $b_d\in (-s_c,0)$ such that
\[
	\Area_{\sigma}{\bigl(D(s_c,b_d)\bigr)}=\Area_{\sigma}{\bigl(D(1,d)\bigr)}.
\]
Then, 
\[
	\Area_{\sigma}{\bigl(D(s_c,-s_c)\bigr)}>\Area_{\sigma}{\bigl(D(s_c,b_d)\bigr)}=\Area_{\sigma}{\bigl(D(1,d)\bigr)}.
\]
Hence the subset $\alpha(s_c,-s_c)$ is displaceable from $\alpha(1,d)$ in the annulus $(-1,1)\times\RR/2\pi\ZZ$.
By \eqref{eq:reduction},
$L_d=\iota\left(\tau^{-1}\bigl(\alpha(1,d)\bigr)\right)\subset M$.
We note that $\iota\left(\tau^{-1}\bigl(\alpha(s_c,-s_c)\bigr)\right)=(\Phi_1^{s_c})^{-1}(0,-s_c)\cap M$.
Therefore, applying Proposition \ref{prop:lift} as above, we can prove that $(\Phi_1^{s_c})^{-1}(0,-s_c)$ is displaceable from $L_d$ in $S^2\times S^2$.
Since $L_d$ is $\zeta_d$-superheavy, Theorem \ref{prop:shv is hv} implies that the fiber $(\Phi_1^{s_c})^{-1}(0,-s_c)$ is not $\zeta_d$-superheavy.
This completes the proof of Theorem \ref{thm:between}.
\end{proof}

Now we prove Theorem \ref{intro memo} stated in Section \ref{sec:results}

\begin{proof}[Proof of Theorem \ref{intro memo}]
Let $c\in [-1,-1/2]$.
We set
\[
	Z_c=(\Phi_1^{s_c})^{-1}(0,-s_c)=\iota\left(\tau^{-1}\bigl(\alpha(s_c,-s_c)\bigr)\right)\cup\{(N,S),(S,N)\},
\]
where $s_c=-\cos{\left(\pi\Area_{\sigma}{\bigl(D(1,c)\bigr)}\right)}$.
If $c\neq -1$, then $Z_c$ is homeomorphic to the doubly pinched torus $Z$
(Note that the points $(N,S),(S,N)$ correspond to the pinched points).

By Theorem \ref{thm:between}, the subset $Z_c$ of $S^2\times S^2$ is $\zeta_d$-superheavy for any $d\in [-1,c]$.
In particular, Theorem \ref{prop:shv is hv} implies that $Z_c$ is non-displaceable from the $\zeta_{-1}$-superheavy subset $Z_{c'}$ for any $c'\in [-1,-1/2]$ and from the $\zeta_d$-superheavy subset $L_d$ for any $d\in [-1,c]$.
Moreover, we have shown that $Z_c$ is displaceable from $L_d$ for any $d\in (c,-1/2]$ in the proof of Theorem \ref{thm:between}.
\end{proof}

\section*{Acknowledgments}

The authors would like to thank Professors Michael Entov, Kaoru Ono, and Leonid Polterovich for some comments.
Especially, they thank Michael and Kaoru for suggesting Proposition \ref{simple qs deha phv ha hv} and Corollary \ref{existence of hv fiber in simple case} and for giving some advice on notions, respectively.
They also thank Renato Vianna for recommending the first author to read papers on semi-toric geometry.


\bibliographystyle{amsart}

\end{document}